\documentclass[12pt]{article}
\usepackage[margin=1in]{geometry}
\usepackage{amsmath}
\usepackage{amssymb}
\usepackage{amsthm}
\usepackage{ dsfont }
\usepackage{enumitem}
\usepackage{graphicx}
\usepackage{authblk}
\author[ ]{Aaron Berger}

\title{Progressions and Paths in Colorings of $\Z$}
\affil[ ]{Yale University \\ aaron.berger@yale.edu}
\date{}                     
\newcommand{\bb}{\mathbb}

\newcommand{\Z}{\bb Z}

\newcommand{\N}{\bb N}

\newcommand{\ord}{\mathrm{ord}}

\renewcommand{\mod}{\mathrm{mod}~}

\newtheorem{theorem}{Theorem}[section]

\newtheorem{corollary}[theorem]{Corollary}

\newtheorem{conjecture}[theorem]{Conjecture}
\newtheorem{counterexample}[theorem]{Counterexample}

\begin{document}
	\maketitle
	\begin{abstract}
		A \textit{ladder} is a set $S \subseteq \Z^+$ such that any finite coloring of $\Z$ contains arbitrarily long monochromatic progressions with common difference in $S$. Van der Waerden's theorem famously asserts that $\Z$ itself is a ladder. We also discuss variants of ladders, namely \textit{accessible} and \textit{walkable} sets, which are sets $S$ such that any coloring of $\Z$ contains arbitrarily long (for accessible sets) or infinite (for walkable sets) monochromatic sequences with consecutive differences in $S$. We show that sets with upper density 1 are ladders and walkable. We also show that all directed graphs with infinite chromatic number are accessible, and reduce the bound on the walkability order of sparse sets from 3 to 2, making it tight.
	\end{abstract}
	\section{Introduction}
	In 1927, van der Waerden proved his famous theorem concerning arithmetic progressions in finite colorings of $\Z$, which asserts that any finite coloring of $\Z$ contains arbitrarily long arithmetic progressions. Brown, Graham and Landman study variants of this result by considering other classes of sequences and whether these classes must also appear in finite colorings of $\Z$ \cite{BGL99}. One such class that they study are subsets of the set of arithmetic progressions whose common differences all lie in some set $S \subseteq \Z^+$. When any finite coloring of $\Z$ contains arbitrarily long arithmetic progressions whose common difference is in $S$, $S$ was said to be ``large,'' though such sets are now called \textbf{ladders} \cite{GRS16}. Another related class of sequences are walks, defined over some set $S$, which are sequences of integers whose consecutive differences are in $S$. If any finite coloring of $\Z$ contains arbitrarily long walks over $S$, then $S$ is said to be \textbf{accessible} (see \cite{jungic2005conjecture,landman2007avoiding,landman2010avoiding}). Finally, when any finite coloring of $\Z$ contains walks of infinite length over $S$, $S$ is said to be \textbf{infinitely walkable}. 
	
	Instead of asking whether these properties hold for arbitrary finite colorings of $\Z$, we may ask for a set $S$ whether the above properties hold for any $k$-coloring of $\Z$ for fixed $k$. This leads to the analagous notions of \textbf{$k$-ladders}, \textbf{$k$-accessible} sets, and \textbf{$k$-walkable} sets. 
	
	This paper is largely motivated by the work of Guerreiro, Ruzsa, and Silva in \cite{GRS16}, and answers three conjectures posed in that paper as well as one conjecture from \cite{BGL99}.
	
	\subsection{Organization of Paper}
	Section \ref{Counterexamples Section} presents a selection of known conditions that are necessary or sufficient for a set to be a ladder, and counterexamples to their converses. Some of these counterexamples were either not stated or unknown in the literature, and Counterexample \ref{Countability of Polynomials} in particular answers a conjecture of \cite{GRS16}. Section \ref{Density} examines density results for ladders and outlines further possibilities for research. Section \ref{Section 3} deals with accessible and walkable sets.
	
	\section{Ladders}
	Perhaps the most natural question to ask regarding ladders is how to determine whether a set is or is not a ladder. However, as of yet, there currently exists no easily checked condition that is necessary and sufficient for a set to be a ladder.
	\begin{theorem}[Brown, Graham, Landman \cite{BGL99}]\label{Modular Restrictions}
		A set $S \subseteq \Z^+$ is a ladder if and only if $S \cap n\Z$ is a ladder for every $n$.
	\end{theorem}
	The proof is reasonably simple and informative. 
	\begin{proof}
		If any such intersection were a non-ladder, then there would be some coloring $\chi$ on which all arithmetic progressions over $S \cap n\Z$ had length bounded by some constant. Then consider the coloring $\pi: x \mapsto x~ \mod n$. The Cartesian product $\chi \times \pi$ produces a coloring of $\Z$ whose monochromatic arithmetic sequences over $S$ lie in the subset $S \cap n\Z$, and are therefore of bounded length. We then conclude $S$ is not a ladder.
		
		The other direction is seen immediately by taking $n = 1$.
	\end{proof}
	

	 
	\subsection{Known Results and New Counterexamples}\label{Counterexamples Section}
	We begin with some known results that give necessary or sufficient conditions for sets to be ladders, and provide counterexamples to their converses, some of which (for example, the counterexample to the converse of Theorem \ref{Polynomial}) were not in previous literature.
	\begin{theorem}[Brown, Graham, Landman \cite{BGL99}]\label{Exponential} If a set $S = \{s_1,s_2,\dots\}$ satisfies $s_{i+1} \geq (1+\epsilon) s_i$ for all $i$ and some $\epsilon > 0$, then $S$ is not a ladder.
	\end{theorem}
	\begin{counterexample}[Converse of Theorem \ref{Exponential}]
		The set of odd integers provides a counterexample to the converse of this theorem. It is a non-ladder by Theorem \ref{Modular Restrictions}.
	\end{counterexample}
	It may also be tempting to hypothesize that sets with exponential growth rates are non-ladders, but this is not the case. Indeed, we will see as a consequence of Theorem \ref{combinatorial cube} that for any $f: \Z\to\Z$, there exists a ladder $S = \{s_1,s_2,\dots\}$ such that $s_i > f(i)$ for all $i$.

	\begin{theorem}[Brown, Graham, Landman \cite{BGL99}]\label{Complement}
		The complement of a non-ladder is a ladder.
	\end{theorem}
	\begin{counterexample}[Converse of Theorem \ref{Complement}]
		As we will see in Theorem \ref{Polynomial}, the set $\{4n^2\}_{n \in \N}$ is a ladder, and its complement contains $\{2n^2\}_{n \in \N}$, which is also a ladder by the same theorem.
	\end{counterexample}
	\begin{theorem}[Brown, Graham, Landman \cite{BGL99}]\label{Polynomial}
		Let $P$ be some polynomial with integer coefficients such that $P(0) = 0$. Then if a set $S$ contains $P(\Z) \cap \Z^+$ then $S$ is a ladder.
	\end{theorem}
	This theorem follows from an extension of Van der Waerden's theorem to such polynomials due to Bergelson and Leibman in \cite{bergelson1996polynomial}. Again, the converse is not always true. 
	\begin{counterexample}[Converse of Theorem \ref{Polynomial}]\label{Countability of Polynomials}
		By the countability of polynomials with integer coefficients, we can construct a set $S$ that contains one element of $P(\Z) \cap \Z^+$ and excludes one element of $P(\Z) \cap \Z^+$ for each nonconstant polynomial $P$ with integer coefficients. Then neither this set nor its complement contain all values of $P(\Z) \cap \Z^+$ for any $P \in \Z[x]$, yet by Theorem \ref{Complement} at least one of these sets must be a ladder.
	\end{counterexample}
	
	For the next theorem, we introduce the following notation: A \textbf{combinatorial cube} of dimension $k$ is the set of all subset sums of a multiset of cardinality $k$.
	\begin{theorem}[Brown, Graham, Landman \cite{BGL99}]\label{combinatorial cube}
		If a set $S \subset \Z$ contains combinatorial cubes of arbitrarily large dimension, then $S$ is a ladder.
	\end{theorem}
	\begin{counterexample}[Converse of Theorem \ref{combinatorial cube}]
		The set of perfect cubes $\{n^3 : n \in \Z\}$ provides a counterexample to the converse of this statement. Such a set is a ladder by Theorem \ref{Polynomial}. If it contained a combinatorial cube of dimension at least 2, it would contain two elements $a$ and $b$, as well as their sum $a+b$. This, however, would violate Fermat's Last Theorem.
	\end{counterexample}
	
	Theorem \ref{combinatorial cube} also implies that we can construct ladders that are arbitrarily sparse by taking a set of combinatorial cubes that are sufficiently far apart from each other. As such, it seems unlikely that any simple density notion can be a necessary and sufficient condition for a set to be a ladder.
	\subsection{Density 1 Sets are Ladders}\label{Density}
	\begin{theorem}
		Any set $S \subset \Z^+$ with upper density 1 is a ladder.
	\end{theorem}
	\begin{proof}
		We show that any set with upper density 1 contains arbitrarily long sequences of the form $\{x,2x,3x,\dots\}$. Each of these sequences is a combinatorial cube, therefore this would imply that such a set is a ladder (Theorem \ref{combinatorial cube}). Recall that a set has upper density 1 in $\Z^+$ if
		$$
		\limsup_{n \to \infty} \frac{|S\cap [1,n]|}{n} = 1.
		$$

	Then for any $n$, we can find some $N$ with $|S \cap [1,N]| > N(1-\frac 1 {n^2})$. Now consider the sequences $\{x,2x,\dots,nx\}$ for $x \in [1,\frac Nn]$. Assume for the sake of contradiction that $S$ contains at most $n-1$ elements in each of these sequences. Then the complement of $S$ contains at least one element from each sequence. We note that each number $t$ can appear in at most $n$ sequences. Otherwise, we would have $t$ appearing in the $k^\text{th}$ spot in two distinct sequences $\{x,2x,\dots,nx\}$ and $\{y,2y,\dots,ny\}$ for some $k$ by pigeonhole principle. This would imply $t = kx = ky$, so $x = y$ and the two sequences are precisely the same, which is a contradiction. We then conclude that the complement of $S$ must contain at least one element for every $n$ sequences constructed above, so its size is at least $N/n^2$, contradicting the density assumption above. Thus $S$ contains $\{x,2x,\dots,nx\}$ for some $x$ in this interval. Since $n$ was chosen arbitrarily, the result follows.
	\end{proof}
	\begin{corollary}
		Let $S$ be a set of the form $P(\Z) \cap \Z^+$ for some $P \in \Z[x]$ of degree at least 2 such that $P(0) = 0$. Then $S$ and its complement are both ladders.
	\end{corollary}
	 This negatively resolves a question in \cite{BGL99} as to whether sets of the form $S \cap P(\Z)$ are ladders for all ladders $S$ and any non-linear $P$ as above. By taking $S$ to be the complement of $P(\Z)$ for any choice of $P$, we see that this intersection is empty and thus not a ladder.

	To conclude the section, we present a conjectural density condition for a set to be a ladder. From Theorem \ref{Modular Restrictions} we see that  $\Z \setminus n\Z$ is a non-ladder, which means that we can construct ladders with density $1-\epsilon$. The following conjecture asserts that this ``modular restriction'' is the only such obstacle to a density condition. Specifically,
	\begin{conjecture}
		Any set $S \subset \Z$ with positive relative upper density in each subgroup $n\Z$ is a ladder, where the {relative upper density} of $S$ in a subgroup $n\Z$ is defined as 
		$$
		\limsup_{k \to \infty} \frac{\left|S \cap \{n,2n,\dots,kn\}\right|}{k}.
		$$
	\end{conjecture}
	Any partial results or weaker variants would still be quite interesting.

	\section{Accessible and Walkable sets}\label{Section 3}
	We now define accessible and walkable sets, which are two commonly-studied variants of ladders \cite{GRS16,jungic2005conjecture, landman2007avoiding, landman2010avoiding}. 
		For a set $S \subset \Z^+$, define its \textbf{distance graph} $G(S) = (V,E)$ with $V = \Z$ and $E = \{(v_1,v_2) \in V \times V \mid |v_1-v_2| \in S\}$.
		A \textbf{walk} over a set $S$ is a sequence $\{a_1,a_2,\dots\}$, of either finite or infinite length, such that for all $i$, $a_{i+1}-a_i \in S$. Equivalently, it is the set of vertices of some path in $G(S)$.
		We say a set $S \subseteq Z^+$ is \textbf{accessible} if any finite coloring of $\Z$ admits arbitrarily long monochromatic walks over $S$.
	We say a set $S \subset \Z^+$ is \textbf{$k$-walkable} if for any $k$-coloring of $\Z$, there are infinitely long monochromatic walks over $S$. A set that is $k$-walkable for all $k$ is called infinitely walkable. (Note the slight distinction between accessible and walkable sets.)
	
	We briefly note the connections between these types of sets and ladders. It is clear that all ladders are accessible, but Jungi\'c provides an example of an accessible sequence that is not a ladder \cite{jungic2005conjecture}. It is not immediately obvious whether a ladder should be infinitely walkable, or vice-versa, however the authors of \cite{GRS16} provide examples of ladders that are not infinitely walkable and of infinitely walkable sets that are not ladders. 
	
	The following result parallels our earlier density result regarding ladders.
	\begin{theorem}\label{Walkable Density 1}
		Any set $S \subset \Z^+$ with upper density 1 is infinitely walkable.
	\end{theorem}
	\begin{proof}
		We will construct an infinite set $H$ such that $H-H \subseteq S$. This will immediately imply $S$ is infinitely walkable (see \cite{GRS16}). We proceed inductively, by constructing a sequence of sets $H_1 \subset H_2 \subset H_3 \subset \cdots$ such that $H_i - H_i \subset S$ for all $i$.
		
		To begin, take $H_1 = \{h_1\}$ for some $h_1 \in S$. Such an element must exist by the density of $S$. Now, say we have $H_k = \{h_1,\dots,h_k\}$ such that $H_k-H_k \subseteq S$. Then fix $n > h_k$ and consider the sets $n-H_k, 2n-H_k,3n-H_k,\dots$. Each of these sets $tn-H_k$ lies in the interval $((t-1)n,tn)$ and so they are mutually disjoint. Assume for the sake of contradiction that none of these sets is contained entirely in $S$. Then in each interval $((t-1)n,tn)$, $S$ is missing at least one element, and so for all $N > 2n(n-1)$ the density of $S$ on an interval $[1,N]$ is bounded by $$\frac {n-1}N \left\lceil{\frac Nn}\right\rceil < \frac {n-1}{n}+\frac{n-1}{N} < \frac{2n-1}{2n},$$ contradicting the upper density assumption on $S$.
		
		Then, by contradiction, we have some $tn$ such that all of its differences with elements of $H$ are in $S$. Then we have $H \cup \{tn\} - H \cup \{tn\} \subseteq S$, so we let $H_{k+1} = H_k \cup \{tn\}$. 
		
		Finally, take $H = \bigcup_{i = 1}^\infty H_i$. This forms an infinite subset of $S$. Any two elements lie in some $H_k$ for sufficiently large $k$, and so their difference lies in $S$. Thus, we conclude that $H-H \subset S$, completing the proof.
	\end{proof}
	\subsection{Walkability Order}
	For sets that are not infinitely walkable we define the order of a set as follows. $$\ord(S) := \sup\{k \mid S \text{ is $k$-walkable}\}.$$
	We prove the following theorem concerning the order of sets whose elements grow quickly, improving on a similar result of Guerreiro, Ruzsa, and Silva in \cite{GRS16} by nearly a factor of two.
	\begin{theorem}\label{Walkability Order}
		Say $S \subset \Z^+$ and $S = \{s_1,s_2,\dots\}$ such that $\liminf \{s_{i+k}-s_i\}$ is infinite. Then $\ord(S) \le k+1$.
	\end{theorem}
	For comparison, see the proof of Theorem 9 in \cite{GRS16}.
	\begin{proof}
		We construct a $(k+2)$-coloring of $\Z^+$ such that every monochromatic walk over $S$ is of finite length. First, partition $\Z$ into intervals in the following manner:
		\begin{enumerate}
			\item Set $I_1 = \{1\}$.
			\item For all $t > 1$, $I_t$ begins immediately after $I_{t-1}$ ends, and $|I_t| = s_N$, where $N$ is chosen such that for all $n > N$, $s_{n+k}-s_{n} > \sum_{i = 1}^{t-1} |I_i|$. Such an $N$ must exist by the assumptions on $S$.
		\end{enumerate}
		With this partition, we have the following fact: An element $x$ in interval $I_t$ is adjacent to at most $k$ elements in the set $I_1 \cup \dots \cup I_{t-2}$. This follows from the fact that for $x$ to be adjacent to an element $y$ of this set, these two elements must differ by some element greater than $s_n > |I_{t-1}| = S_N$ as above. Then $x-s_{n+k} < x-s_n-\sum_{i = 1}^{t-2} |I_i| < 0$. And so there are at most $k$ values of $s_i$ such that $x-s_i \in I_1 \cup \dots \cup I_{t-2}$.
		
		We now color the integers using the elements of $[k+2]$. In each interval $I_t$, we will restrict ourselves to using the $k+1$ elements of $[k+2]$ not equivalent to $t ~\mod k+2$. The coloring proceeds as follows: each element $t \in I_t$ is adjacent to at most $k$ elements in $I_1 \cup \dots \cup I_{t-2}$ from above. We have $k+1$ choices for colors of elements in this interval, so choose a color for $t$ that is not equal to any of the colors of these neighbors, if they exist. Now, in this coloring, no element is adjacent to any element that is greater than one interval away. Then any monochromatic walk over $S$ contains elements in a consecutive set of intervals. If this consecutive set contained $k+2$ intervals, it would have to contain an interval with no elements of its color, which is impossible. Thus each monochromatic walk over $S$ in this coloring is contained in a set of at most $k+1$ finite intervals, and is therefore finite, which completes the proof.
	\end{proof}
	As a specific useful case of this theorem, we present the following corollary.
	\begin{corollary}\label{Walkability Order 2}
		 Let $S = \{s_i\}$ with $\liminf\{s_{i+1}-s_i\} = \infty$. Then $\ord(S) \leq 2$.
	\end{corollary}
	This answers a problem of \cite{GRS16} that asks for the order of the set of squares, which the authors of the conjecture showed to be at least 2. Corollary \ref{Walkability Order 2} shows that the order of this set must then equal 2. Moreover, this example shows that given only the above assumptions on $S$, the bound given by Theorem \ref{Walkability Order} on the walkability order of $S$ is tight.
	
	\subsection{Accessibility of General Directed Graphs}
	We can also study accessibility and walkability through the distance graph $G(S)$. Guerreiro, Ruzsa, and Silva show that a set $S$ is accessible if and only if $G(S)$ has infinite chromatic number, that is, any finite coloring of $G(S)$ contains a pair of adjacent vertices of the same color \cite{GRS16}. This proof extends readily to all acyclic directed graphs, but is an open question for general directed graphs. We state their result here and then extend it to arbitrary directed graphs.
	\begin{theorem}[Guerreiro, Ruzsa, Silva \cite{GRS16}]\label{Accessible Graph}
		Let $G$ be an acyclic directed graph. Then $G$ has infinite chromatic number if and only if $G$ is accessible.
	\end{theorem}
	We prove the following extension of this theorem.
	\begin{theorem}\label{General Directed Graphs}
		Let $G$ be a directed graph with no loops. Then $G$ has infinite chromatic number if and only if $G$ is accessible.
	\end{theorem}
	\begin{proof}
		Recall that $G$ having infinite chromatic number means that any finite coloring admits monochromatic paths of a single edge, whereas accessible means that any finite coloring admits arbitrarily long monochromatic paths. Then accessibility implies infinite chromatic number.
		
	For the other direction, impose an arbitrary ordering on the vertices $V$ and then partition the edges of $G$ into two sets $E_1$ and $E_2$, where $E_1 = \{(v_1,v_2) \mid v_1 < v_2\}$ and $E_2$ its complement. Then it is clear that the two graphs $(V,E_1)$ and $(V,E_2)$ are acyclic, and both are subgraphs of $G$. Now we show that at least one of these graphs has infinite chromatic number. Assuming otherwise, there would exist some finite colorings $\chi_1$ and $\chi_2$ such that for all $(x,y) \in E_1$, $\chi_1(x) \neq \chi_1(y)$, and similarly for $(x,y) \in E_2$, $\chi_2(x) \neq \chi_2(y)$. Then consider the Cartesian product $\chi_1 \times \chi_2$, which is again a finite coloring. Any pair of adjacent vertices in $G$ must be connected by an edge in either $E_1$ or $E_2$, and therefore differ in at least one coordinate of their color in this coloring. Thus in this finite coloring of $G$, no pair of adjacent vertices share the same color, which contradicts the assumption that $G$ has infinite chromatic number. Then, by contradiction, one of these two acyclic subgraphs of $G$ also has infinite chromatic number. Then any coloring of the vertices of $G$ is a coloring of this subgraph, which by Theorem \ref{Accessible Graph} contains arbitrarily long monochromatic paths.
	\end{proof}
	\section{Further Work}
	This paper answers a number of questions from \cite{GRS16}, but not all of them. One of the very interesting open problems is the question of whether a $2$-ladder is necessarily a ladder. We offer a density variant as well: Say $S \subseteq \Z$ is $\alpha$-Szemer\'edi if any $X \subseteq \Z^+$ with upper density $\alpha$ contains arbitrarily long arithmetic progressions with common difference in $S$. The conjecture states that for any $S$, $\inf\{\alpha: S \text{ is $\alpha$-Szemer\'edi}\} \in \{0,1\}$.
	
	\section{Acknowledgements}
	This research was carried out at the Duluth REU under the supervision of Joe Gallian. Duluth REU is supported by the University of Minnesota Duluth and by grants NSF-1358659 and NSA H98230-16-1-0026. Special thanks to Eric Riedl and Joe Gallian for editing help. Thanks as well to the UMD car rental program, because I was surprisingly productive during the long walks back to my apartment after returning their cars.
	\bibliographystyle{acm}
	\bibliography{Ladder_Bibliography}
\end{document}